\documentclass[a4paper,11pt]{article}
\usepackage{amssymb}
\usepackage{amsmath}
\usepackage{amsthm}
\usepackage{graphicx}

\newtheorem{thm}{Theorem}[section]
\newtheorem{lem}[thm]{Lemma}

\newtheorem{rmk}{Remark}

\numberwithin{equation}{section}

\newcommand{\va}{\varphi}
\newcommand{\ppp}{\partial}
\newcommand{\ddda}{d^{\alpha}_t}
\newcommand{\pppa}{\partial_t^{\alpha}}

\newcommand{\N}{\mathbb{N}}

\newcommand{\OOO}{\Omega}
\newcommand{\CCC}{_{0}C^1[0,T]}

\newcommand{\Halp}{H_{\alpha}(0,T)}
\newcommand{\hhalf}{\frac{1}{2}}

\allowdisplaybreaks

\title{Uniqueness for fractional nonsymmetric diffusion equations and 
an application to an inverse source problem}

\author{
Daijun Jiang\\
School of Mathematics and Statistics,\\
Hubei Key Laboratory of Mathematical Sciences, \\
Central China Normal University, Wuhan 430079, China\\ 
email: {\tt jiangdaijun@mail.ccnu.edu.cn}\\
Zhiyuan Li \\
School of Mathematics and Statistics,
Shandong University of Technology, \\
266 Xincunxi Road, Zibo, Shandong 255049, China\\
email: {\tt zyli@sdut.edu.cn}\\
Matthieu Pauron\\
ENS Rennes, 35170 Bruz, France\\
email: {\tt matthieu.pauron@ens-rennes.fr}\\
Masahiro Yamamoto\\
Graduate School of Mathematical Sciences, the University of Tokyo\\
3-8-1 Komaba, Meguro-ku, Tokyo 153-8914, Japan\\
Honorary Member of Academy of Romanian Scientists,\\ 
Splaiul Independentei Street, no 54 \\
050094 Bucharest Romania,
email: {\tt myama@ms.u-tokyo.ac.jp}
}

\date{}

\begin{document}
\maketitle

\begin{abstract}
In this paper, we discuss the uniqueness for solution to time-fractional 
diffusion equation $\pppa (u-u_0) + Au=0$ with the homogeneous Dirichlet boundary condition, where an elliptic operator $-A$ is not necessarily symmetric. 
We prove that the solution is identically zero if its normal 
derivative with respect to the operator $A$ vanishes on an arbitrary small part of the spatial domain
over a time interval.
The proof is based on the Laplace transform and the spectral decomposition, and
is valid for more general time-fractional partial differential equations, including 
those involving non symmetric operators.
\end{abstract}


\section{Introduction and main results}
Throughout this paper, we assume that $T>0$, $0<\alpha<1$, and 
$\Omega\subset\mathbb R^d$ is a bounded domain with sufficiently smooth 
boundary $\partial\Omega$, and let $\nu=(\nu_1,\cdots,\nu_d)$ denote the unit 
outwards normal vector to the boundary $\partial\Omega$. Let the operator $A$ be defined by 
$$
-A\varphi:=\sum_{j,k=1}^d \partial_j(a_{jk}(x) \partial_k \varphi) 
+ \sum_{j=1}^d b_j(x)\ppp_j\va + c(x) \varphi,
\quad \varphi\in \mathcal{D}(A) := H_0^1(\Omega)\cap H^2(\Omega),
$$
where the principle part is uniformly elliptic, namely $a_{ij}=a_{ji},
b_j \in C^1(\overline\Omega)$, $1\le i,j\le d$ and $c\in L^{\infty}(\OOO)$, 
satisfy that 
\begin{equation}
\label{condi-elliptic}
a_0 \sum_{j=1}^d \xi_j^2 \le \sum_{j,k=1}^d a_{jk}(x) \xi_j\xi_k, \quad x\in\overline\Omega,\ \xi\in\mathbb R^d,
\end{equation}
where $a_0>0$ is a constant independent of $x,\xi$.
Here we set the normal derivative with respect to the operator $A$ as 
$$
\partial_{\nu_A} u = \sum_{i,j=1}^d a_{ij}\nu_i \partial_j u.
$$
for any $ u \in H^{3/2}(\partial\Omega)$.
We define the time-fractional derivative $\pppa$.
First let $H^{\alpha}(0,T)$ be the fractional Sobolev space
with the norm
$$
\vert v\vert_{H^{\alpha}(0,T)} 
= \left( \vert v\vert_{L^2(0,T)}^2 
+ \int^T_0 \int^T_0 \frac{\vert v(t) - v(s)\vert^2}
{\vert t-s\vert^{1+2\alpha}} dsdt \right)^{\hhalf}
$$
(e.g., Adams \cite{Ad}).
We set 
$$
\Halp =
\begin{cases}
\{ v \in H^{\alpha}(0,T); \, v(0) = 0\},  &\hhalf < \alpha < 1, \\
\left\{ v \in H^{\hhalf}(0,T);\, \int^T_0 \frac{\vert v(t)\vert^2}{t} dt
< \infty\right\}, &\alpha = \hhalf, \\
H^{\alpha}(0,T),  &0 < \alpha < \hhalf, 
\end{cases}
$$
and
$$
\vert v\vert_{\Halp} = 
\begin{cases}
\vert v\vert_{H^{\alpha}(0,T)}, & 0 < \alpha < 1, \, \alpha \ne \hhalf, \\
\left( \vert v\vert^2_{H^{\hhalf}(0,T)} + \int^T_0 \frac{\vert v(t)\vert^2}{t} dt\right)^{\hhalf}, & 
\alpha = \hhalf.
\end{cases}
$$
We set 
$$
J^{\alpha}v(t) = \frac{1}{\Gamma(\alpha)}\int^t_0 (t-s)^{\alpha-1}
v(s) ds.
$$
Then it is known 
$$
J^{\alpha}L^2(0,T) = \Halp, \quad 0 < \alpha \le 1
$$
and there exists a constant $C>0$ such that 
$$
C^{-1}\vert J^{\alpha}v\vert_{\Halp} \le \vert v\vert_{L^2(0,T)}
\le C\vert J^{\alpha}v\vert_{\Halp}
$$
for $v \in L^2(0,T)$ (Gorenflo, Luchko and Yamamoto \cite{GLY}).
We define the time-fractional derivative $\pppa$ in $\Halp$ by 
$$
\pppa v = (J^{\alpha})^{-1}v, \quad v \in \Halp.
$$

\begin{rmk}
\label{rem-caputo}
We define the Caputo derivative
$$
\ddda v(t) = \frac{1}{\Gamma(1-\alpha)}\int^t_0 (t-s)^{-\alpha}
\frac{dv}{ds}(s) ds
$$ 
and we consider $\ddda$ for $v \in {\CCC} := \{ v \in C^1[0,T];\,
v(0) =  0\}$.  Regarding $\ddda$ as an operator with the domain 
$\CCC$, we can see that the minimum closed extension of $\ddda$ coincides 
with $\pppa$ (Kubica, Ryszewska and Yamamoto \cite{KRY}).
\end{rmk}

We consider
\begin{equation}
\label{eq-gov}
\left\{
\begin{alignedat}{2}
& \partial_t^{\alpha} (u - u_0) + A u = 0 &\quad& \mbox{in $\Omega\times(0,T)$,}\\
&u(x,\cdot)-u_0(x)\in \Halp, &\quad&\mbox{for almost all $x\in\Omega$,}\\
&u(x,t)=0, &\quad& \mbox{$(x,t)\in\partial\Omega\times(0,T)$.}
\end{alignedat}
\right.
\end{equation}
We assume that $u_0 \in L^2(\Omega)$.  Then it is known 
(e.g., Kubica, Ryszewska and Yamamoto \cite{KRY}, Kubica and Yamamoto 
\cite{KY}) that there exists a unique solution 
\begin{equation}
\label{rslt-regu}
u \in L^2(0,T; H^1_0(\OOO) \cap H^2(\Omega)) 
\end{equation}
such that $u-u_0 \in H_{\alpha}(0,T;L^2(\OOO))$ and 
$u(\cdot,t) \in H^1_0(\OOO)$ for almost all $t \in (0,T)$.
We refer also to Gorenflo, Luchko and Yamamoto \cite{GLY},
Zacher \cite{Za}.  By \eqref{rslt-regu}, we see that $\ppp_{\nu_A}u 
\in L^2(\ppp\OOO \times (0,T))$.

We are ready to state the first main result.
\begin{thm}
\label{thm-ucp}
Let $\Gamma \subset \ppp\OOO$ be an arbitrarily chosen 
subboundary and let $T>0$.
For $u_0 \in L^2(\Omega)$, let $u$ be the solution to \eqref{eq-gov}.
If $\ppp_{\nu_A}u = 0$ on $\Gamma \times (0,T)$, then 
$u=0$ in $\OOO \times (0,T)$.
\end{thm}

This uniqueness result is known to be equivalent to the approximate 
controllability for the adjoint system to (1.1) (e.g., Fujishiro and 
Yamamoto \cite{FY}).  For evolution equations with natural number 
order time-derivative, see e.g., Schmidt and Weck \cite{SW}, 
Triggiani \cite{Tr}.  Here we do not discuss about the approximate 
controllability.

In the case where $-A$ is symmetric, that is, $b_j = 0$ for $j=1,..., d$, 
there are several results.  For example, we refer to Sakamoto and Yamamoto 
\cite{SY}.  See also Jiang, Li, Liu and Yamamoto \cite{JLLY} for not 
necessarily symmetric $A$.  The argument in \cite{JLLY} is different from 
ours: \cite{JLLY} 
relies on the transformation of the problem to the determination 
of $u_0$ of the 
corresponding parabolic equation through the Laplace transform.  
In both \cite{JLLY} and \cite{SY}, the condition 
$u=0$ in $\omega \times (0,T)$ with a subdomain $\omega \subset \OOO$, is 
assumed in place of $\ppp_{\nu_A}u = 0$ on $\Gamma \times (0,T)$.
We note that for $\ppp_{\nu_A}u$, we here assume more regularity for the 
initial value, that is, $u_0 \in H^1_0(\OOO)$.

Our proof is based on the spectral property of the elliptic operator 
$A$, while the proof in \cite{JLLY} follows from the uniqueness result
for a parabolic equation:
$$
\left\{ 
\begin{alignedat}{2}
& \ppp_tu + Au = 0, &\quad& \mbox{in $\OOO\times (0,T)$}, \\
& u\vert_{\ppp\OOO\times (0,T)} = 0, &\quad& \ppp_{\nu_A}u\vert_{\Gamma\times (0,T)} = 0.
\end{alignedat}\right.
$$
Indeed our proof of Theorem \ref{thm-ucp} works also for higher-order elliptic 
operator $A$, for example,
\begin{equation}
\label{eq-gov^m}
\pppa u = -\left(-\Delta + \sum_{j=1}^d b_j(x)\ppp_j + c(x)\right)^mu 
\end{equation}
with $m \in \N$ under suitable conditions, but the method in \cite{JLLY}
requires us to prove that if $u$ satisfies \eqref{eq-gov^m} and 
$$
\ppp_{\nu}^ku = 0 \quad \mbox{on $\ppp\OOO\times (0,T)$, $k=0,1,..., 2m-1$,}
$$
then $u=0$ in $\OOO\times (0,T)$, which needs more arguments than our proof
for the case $m>1$.  In the case of $m=1$, this uniqueness follows from 
the well-known unique continuation for a parabolic equation 
(e.g., Isakov \cite{Is}, Yamamoto \cite{Ya}).

As one application of Theorem 1, we show the uniqueness for an inverse source 
problem.  We consider
\begin{equation}
\label{eq-sp}
\left\{
\begin{alignedat}{2}
& \partial_t^{\alpha} y + A y = \mu(t)f(x), &\quad& x\in \OOO,\, 0<t<T, \\
& y(x,\cdot)\in \Halp, &\quad& \mbox{for almost all $x\in\Omega$,}\\
& y(x,t)=0, &\quad& (x,t)\in\partial\Omega\times(0,T).  
\end{alignedat}\right.
\end{equation}
We assume that $f \in L^2(\OOO)$ and $\mu \in L^2(0,T)$.  Then we know 
(e.g., \cite{KRY}) that there exists a unique solution 
$y \in L^2(0,T;H^2(\OOO)\cap H^1_0(\OOO)) \cap H_{\alpha}(0,T;L^2(\OOO))$ to 
(1.4).  Now for given $\mu$, we discuss an inverse source problem of 
determining $f$ in $\OOO$ by $\ppp_{\nu_A}y\vert_{\Gamma \times (0,T)}$.

\begin{thm}
\label{thm-isp}
Let $\Gamma \subset \ppp\OOO$ be an arbitrarily chosen subboundary and 
let $f \in L^2(\Omega)$, $\mu \in C^1[0,T]$, $\not\equiv 0$ in $[0,T]$.
If $\ppp_{\nu_A}y = 0$ on $\Gamma \times (0,T)$, then $f=0$ in $\OOO$.
\end{thm}

In this article, we discuss the determination of initial value and source term 
within the framework of \cite{KRY} which formulates initial boundary value 
problems (1.2) and (1.5) and establishes the well-posedness in fractional 
Sobolev spaces.  In particular, for the first time 
we establish the uniqueness in the 
inverse source problem for (1.5) for general $f\in L^2(\Omega)$ and $\mu 
\in L^2[0,T]$, 
where the time-regularity of the solution $y$ is delicate.
\\

This article is outlined as follows. 
In section \ref{sec-pre}, we first provide several preliminary results from spectral theory and prove some auxiliary lemmas for the formula of the Laplace transform for the fractional derivative $\partial_t^\alpha$ in $H_\alpha(0,T)$, which plays crucial roles in the proof of the main theorem. In section \ref{sec-ucp}, by the Laplace tranform argument, the unique continuation principle in Theorem \ref{thm-ucp} is proved. 
In section \ref{sec-isp}, based on the unique coninuation principle, from the Duhamel principle, see Lemma \ref{lem-Duhamel}, we finish the proof of Theorem \ref{thm-isp}.
Finally, a concluding remark is given in section \ref{sec-rem}.

\section{Preliminaries}
\label{sec-pre}
\subsection{Well-posedness for the forward problem}
In this part, we are concerned with the wellposedness for the initial-boundary value problem \eqref{eq-gov}. More precisely, we will show that the solution to the problem \eqref{eq-gov} admit an exponential growth, which is essential for carrying out the Laplace transform argument.
\begin{lem}[Coercivity]
\label{lem-coercivity}
For any measurable function $\varphi(\cdot)-c_0\in H_\alpha(0,T)$ with $\alpha\in(0,1)$, one has the coercivity inequality
$$
\frac2{\Gamma(\alpha)} \int_0^t (t-s)^{\alpha-1} \varphi(s)\partial_t^\alpha (\varphi(s) - c_0) d s \ge \varphi^2(t) - c_0^2,\quad \mbox{almost all $t\in(0,T)$}.
$$
\end{lem}
\begin{proof}
We divide the proof into two steps. First, we assume $\varphi-c_0\in {_0}C^1[0,T]$, then from Lemma 1 in Alikhanov \cite{Al10}, it follows that
\begin{equation}
\label{esti-Caputo}
\varphi(t) d_t^\alpha \varphi(t) \ge \frac12 d_t^\alpha (\varphi^2)(t),\quad t>0.
\end{equation}
Now noting that $\varphi(0)=c_0$, along with the formula
$$
J^\alpha d_t^\alpha \varphi = \varphi(t) - c_0,\quad t\in(0,T),
$$
we have
$$
\frac2{\Gamma(\alpha)} \int_0^t (t-s)^{\alpha-1} \varphi(s) d_s^\alpha \varphi(s) ds 
\ge \varphi^2(t) - c_0^2,\quad t\in(0,T).
$$
Since $d_t^\alpha c_0 = 0$, we further arrive at the inequality
$$
\frac2{\Gamma(\alpha)} \int_0^t (t-s)^{\alpha-1} \varphi(s) d_t^\alpha (\varphi(s) - c_0) d s \ge \varphi^2(t) - c_0^2,\quad \mbox{almost all $t\in(0,T)$}.
$$
Moreover, from Remark \ref{rem-caputo}, it follows that the Caputo derivative $d_t^\alpha$ concides with the fractional derivative $\partial_t^\alpha$ under the domain ${_0}C^1[0,T]$, and then we see that
$$
\frac2{\Gamma(\alpha)} \int_0^t (t-s)^{\alpha-1} \varphi(s) \partial_t^\alpha (\varphi(s) - c_0) d s \ge \varphi^2(t) - c_0^2,\quad \mbox{almost all $t\in(0,T)$}.
$$

Now for the case $\varphi-c_0\in H_\alpha(0,T)$, letting $\psi:=\varphi - c_0\in H_\alpha(0,T)$, it is equivalent to prove the inequality
$$
\frac2{\Gamma(\alpha)} \int_0^t (t-s)^{\alpha-1} (\psi(s)+c_0)\partial_t^\alpha \psi(s) d s \ge (\psi(t)+c_0)^2 - c_0^2,\quad \mbox{almost all $t\in(0,T)$}.
$$
From the argument used in \cite{GLY}, we see that $\overline{{_0}C^1[0,T]}^{H_\alpha(0,T)} = H_\alpha(0,T)$, hence we can choose $\psi_n\in {_0}C^1[0,T]$ and $\psi_n$ tends to $\psi$ under the norm of $H^\alpha(0,T)$ as
$n\to\infty$. Then from the conclusion in the first step, we see that
\begin{equation}
\label{ineq-coe-appro}
\frac2{\Gamma(\alpha)} \int_0^t (t-s)^{\alpha-1} (\psi_n(s)+c_0)\partial_t^\alpha \psi_n(s) d s \ge (\psi_n(t)+c_0)^2 - c_0^2,\quad \mbox{almost all $t\in(0,T)$}.
\end{equation}
Since $\psi_n\to\psi$ in $H^\alpha(0,T)$ as $n\to\infty$, hence $\psi_n(t)\to\psi(t)$ for almost all $t\in(0,T)$, we see that the right-hand side of the above inequality tends to $(\psi(t)+c_0)^2 - c_0^2$ for almost all $t\in (0,T)$. 

Now we evaluate
\begin{align*}
&\int_0^T \left|\int_0^t (t-s)^{\alpha-1}(\psi_n(s)+c_0)\partial_t^\alpha \psi_n(s) ds - \int_0^t (t-s)^{\alpha-1}(\psi(s)+c_0)\partial_t^\alpha \psi(s) ds\right|dt 
\\
\le&\int_0^T \left|\int_0^t (t-s)^{\alpha-1} \left((\psi_n(s) - \psi(s))\partial_t^\alpha \psi_n(s)\right) ds\right|dt 
\\
&+\int_0^T \left|\int_0^t (t-s)^{\alpha-1}(\psi(s)+c_0)\left( \partial_t^\alpha \psi_n(s) ds - \partial_t^\alpha \psi(s)\right) ds\right|dt
=: I_{n1} + I_{n2}.
\end{align*}
By the use of the Fubini lemma, $I_{n1}(t)$ can be further rewritten by
\begin{align*}
I_{n1}(t)
\le&\int_0^T \left(\int_s^T (t-s)^{\alpha-1}  dt  \right) \left|(\psi_n(s) - \psi(s))\partial_t^\alpha \psi_n(s)\right| ds \\
\le &\frac{T^\alpha}{\alpha}\int_0^T \left|(\psi_n(s) - \psi(s))\partial_t^\alpha \psi_n(s)\right| ds.
\end{align*}
From H\"{o}lder's inequality, by a direct calculation, the last integration on the right hand side of the above estimates can be evaluated by
\begin{align*}
\int_0^T \left|(\psi_n(s) - \psi(s))\partial_t^\alpha \psi_n(s)\right| ds
\le &\|\psi_n-\psi\|_{L^2(0,T)} \left(\int_0^T \Big|\partial_s^\alpha \psi_n(s)\Big|^2 ds\right)^{\frac12}
\\
\le & \|\psi_n-\psi\|_{L^2(0,T)} (\|\psi\|_{H^\alpha(0,T)} + 1) 
\to 0, \quad\mbox{ as $n\to\infty$.}
\end{align*}
Similarly, we see that
\begin{align*}
I_{n2} 
\le& \frac{T^\alpha}{\alpha} \int_0^T \left|(\psi(s) + c_0)\partial_t^\alpha (\psi_n(s) - \psi(s))\right| ds
\\
\le& \frac{T^\alpha}{\alpha} \|\psi + c_0\|_{L^2(0,T)} \left( \int_0^T \left|\partial_t^\alpha (\psi_n(s) - \psi(s))\right|^2 ds \right)^{\frac12} 
\\
\le& \frac{T^\alpha}{\alpha} \|\psi + c_0\|_{L^2(0,T)} \|\psi_n - \psi\|_{H^\alpha(0,T)}
\to 0,\quad \mbox{as $n\to\infty$}.
\end{align*}
Consequently, we see that
$$
\int_0^t (t-s)^{\alpha-1} (\psi_n(s)+c_0)\partial_t^\alpha \psi_n(s) ds
\quad
\mbox{tends to }
\quad
\int_0^t (t-s)^{\alpha-1} (\psi(s)+c_0)\partial_t^\alpha \psi(s) ds
$$
for almost all $t\in(0,T)$. Now letting $n\to\infty$ on both sides of the inequality \eqref{ineq-coe-appro}, we find
$$
\frac2{\Gamma(\alpha)} \int_0^t (t-s)^{\alpha-1} (\psi(s)+c_0)\partial_t^\alpha \psi(s) d s \ge (\psi(t)+c_0)^2 - c_0^2,\quad \mbox{almost all $t\in(0,T)$}.
$$
We finish the proof of the lemma by changing $\psi(t) + c_0$ back to $\varphi(t)$.
\end{proof}

\begin{lem}
\label{lem-analy}
Let $u_0\in L^2(\Omega)$, then the unique solution $u:(0,T)\to H^2(\Omega)$ is $t$-analytic and can be analytically extended to $(0,\infty)$. Moreover, there exists a constant $C>0$ such that 
$$
\|u(\cdot,t)\|_{L^2(\Omega)} \le Ce^{Ct}\|u_0\|_{L^2(\Omega)},\quad t>0.
$$ 
\end{lem}
\begin{proof}
For the proof of the $t$-analyticity of the solution, one can refer to Sakamoto and Yamamoto \cite{SY}, and Li, Huang and Yamamoto \cite{LHY}. It is sufficient to show the solution $u$ admits an exponential growth.  For this, we multiply $u - u_0$ on both sides of the equation \eqref{eq-gov} to derive that
$$
\langle\partial_t^\alpha (u - u_0), u\rangle_{L^2(\Omega)} + \langle Au,u\rangle_{L^2(\Omega)} = 0.
$$
Now multiplying $J^\alpha$ on both sides of the above equation, noting Lemma \ref{lem-coercivity} and from integration by parts, we see that
\begin{align*}
&\frac12 \|u\|_{L^2(\Omega)}^2 - \frac12\|u_0\|_{L^2(\Omega)}^2 
+ J^\alpha\left(\int_\Omega a_{ij}(x) \partial_iu(x,t) \partial_ju(x,t) dx 
\right)
\\
\le& J^\alpha \left( \int_\Omega (B(x)\cdot\nabla u(x,t)+c(x)u(x,t))(u(x,t)) 
dx\right).
\end{align*}
From the ellipticity \eqref{condi-elliptic} and the Cauchy-Schwarz 
inequality, for a sufficiently small $\varepsilon>0$, we can further derive 
$$
\frac12 \|u\|_{L^2(\Omega)}^2 - \frac12\|u_0\|_{L^2(\Omega)}^2 
+ J^\alpha (\|u(\cdot,t)\|_{H^1(\Omega)}^2)
\le J^\alpha \left( \varepsilon \|u(\cdot,t)\|_{H^1(\Omega)}^2 
+ \frac{C}{\varepsilon} \|u(\cdot,t)\|_{L^2(\Omega)} \right).
$$
By taking $\varepsilon>0$ small enough, we have
$$
\|u\|_{L^2(\Omega)}^2 + J^\alpha (\|u(\cdot,t)\|_{H^1(\Omega)}^2)
\le C\|u_0\|_{H^1(\Omega)}^2 + CJ^\alpha \|u(\cdot,t)\|_{L^2(\Omega)},
$$
which implies
$$
\|u(\cdot,t)\|_{L^2(\Omega)}^2 \le C \|u_0\|_{L^2(\Omega)}^2 + \frac{C}{\Gamma(\alpha)} \int_0^t (t-s)^{\alpha-1} \|u(\cdot,s)\|_{L^2(\Omega)}^2 ds.
$$
Therefore we conclude from the general Gronwall inequality that
$$
\|u(\cdot,t)\|_{L^2(\Omega)} \le Ce^{Ct}\|u_0\|_{L^2(\Omega)},\quad t\ge0.
$$
Thus the proof of the lemma is complete.
\end{proof}

\subsection{Laplace transform of $\partial_t^\alpha$}
We define the Laplace transform $(Lu)(p)$ by
$$
(Lu)(p) := \int_0^\infty e^{-pt} u(t) dt
$$
for $\Re p>p_0$: some constant.

The formulae of the Laplace transforms for fractional derivatives are well-known. For example,
\begin{equation}
\label{eq-lap-caputo}
L(d_t^\alpha u)(p) = p^\alpha (Lu)(p) - p^{\alpha-1}u(0)
\end{equation}
for $\Re p>p_0$: some constant. The formula \eqref{eq-lap-caputo} is convenient for solving fractional differential equations. However formula \eqref{eq-lap-caputo} requires some regularity for $u$. For instance, $u(0)$ should be apparently defined and \eqref{eq-lap-caputo} does not make a sense for $u\in H^\alpha(0,T)$ with $0<\alpha<\frac12$.

Moreover such needed regularity should be consistent with the regularity which we can prove for solutions to a fractional differential equations. In particular, the regularity for the formula concerning the Laplace transform should be not very strong. Thus on the regularity assumption for the formula like \eqref{eq-lap-caputo}, we have to make adequate assumptions for $u$.

In this section, we state the formula of the Laplace transform for the fractional derivative $\partial_t^\alpha$ in $H_\alpha(0,T)$. We set 
\begin{equation}
\label{defi-V_alpha}
\begin{split}
V_\alpha(0,\infty) := \{ u\in L_{\rm loc}^1(0,\infty); u|_{(0,T)}\in H_\alpha(0,T) \mbox{ for any $T>0$,}
\\
\mbox{ there exists a constant $C=C_u>0$ such that $|u(t)| \le Ce^{Ct}$ for $t\ge0$.} 
\}
\end{split}
\end{equation}
Here we define a set $L_{\rm loc}^1(0,\infty)$ of functions defined in $(0,\infty)$ by
$$
L_{\rm loc}^1(0,\infty) = \{u;u|_{(0,T)}\in L^1(0,T)\mbox{ for any } T>0\}.
$$
Then we can state
\begin{lem}
\label{lem-lap-caputo}
The Laplace transform $L(\partial_t^\alpha u)(p)$ can be defined for $u\in V_\alpha(0,\infty)$ by
$$
L(\partial_t^\alpha u)(p) = \lim_{T\to\infty} \int_0^T e^{-pt} \partial_t^\alpha u(t) dt,\quad p>C_u
$$
and
$$
L(\partial_t^\alpha u)(p) = p^\alpha Lu(p), \quad p>C_u.
$$
\end{lem}
\begin{proof}
We can refer to \cite{KRY} for the proof, but
for completeness, here we provide the proof.
First for $u\in H_\alpha(0,T)$, by Theorem 2.3 in \cite{KRY}, we can see that 
\begin{equation}
\label{eq-Ju}
J^{1-\alpha} u\in H_1(0,T) \subset H^1(0,T),
\end{equation}
and so
$$
D_t^\alpha u = \frac{d}{dt} J^{1-\alpha} u \in L^2(0,T).
$$
Theorem 2.4 from \cite{KRY} yields $\partial_t^\alpha u = \frac{d}{dt} J^{1-\alpha} u$ for $u\in H_\alpha(0,T)$. Let $T>0$ be arbitrarily fixed. Then, in terms of \eqref{eq-Ju}, we integrate by parts to obtain
\begin{align*}
\int_0^T e^{-pt} \partial_t^\alpha u(t) dt 
=& \int_0^T e^{-pt} \frac{d}{dt} (J^{1-\alpha} u)(t) dt\\
=& \left[ J^{1-\alpha} u(t) e^{-pt} \right]_{t=0}^{t=T} + p\int_0^T e^{-pt} J^{1-\alpha} u(t) dt.
\end{align*}
The Sobolev embedding (e.g., \cite{Ad}) yields
$$
H^\alpha(0,T)\subset
\begin{cases}
L^{\frac{2}{1-2\alpha}}(0,T), & \mbox{if } 0<\alpha<\frac12,\\
L^{\frac1\delta}(0,T), & \mbox{with any $\delta>0$ if } \alpha=\frac12,\\
L^\infty(0,T), & \mbox{if }\frac12<\alpha<1.
\end{cases}
$$
First for $0<\alpha<\frac12$, the H\"older inequality implies
\begin{align*}
|J^\alpha u(t)| 
&= \left| \frac1{\Gamma(1-\alpha)} \int_0^t (t-s)^{-\alpha} u(s) ds \right| 
\\
&\le C\left( \int_0^t (|t-s|^{-\alpha})^{\frac{2}{1+2\alpha}} \right)^{\frac{1+2\alpha}{2}} \left( \int_0^t |u(s)|^{\frac{2}{1+2\alpha}} \right)^{\frac{1+2\alpha}{2}}
\\
&\le C\left( t^{\frac1{1+2\alpha}} \right)^{\frac{1+2\alpha}2} \|u\|_{L^{\frac2{1-2\alpha}}(0,T)} \to0
\end{align*}
as $t\to0$. Next let $\alpha=\frac12$. We choose $\delta\in(0,\frac12)$. Setting $p=\frac1{1-\delta}$ and $q=\frac1\delta$, we apply the H\"older inequality to have 
\begin{align*}
|J^{1-\alpha} u(t)| 
\le& C\left( \int_0^t |t-s|^{-\frac12 p} ds \right)^{\frac1p}
\left( \int_0^t |u(s)|^q ds \right)^{\frac1q}
\\ 
\le& C\left( \int_0^t |t-s|^{-\frac12 \frac1{1-\delta}} ds \right)^{1-\delta}
\left( \int_0^t |u(s)|^{\frac1\delta} ds \right)^{\delta} \to 0
\end{align*}
as $t\to0$ by $0<\delta<\frac12$. Finally for $\frac12<\alpha<1$, we have
$$
|J^{1-\alpha} u(t)| \le C\int_0^t |t-s|^{-\alpha} ds \|u\|_{L^\infty(0,T)} \to 0
$$
as $t\to0$. Thus we see that
$$
\lim_{t\to0} J^{1-\alpha} u(t) = 0.
$$
Hence
\begin{align*}
\int_0^T e^{-pt} \partial_t^\alpha u(t) dt
=& \frac{e^{-pT}}{\Gamma(1-\alpha)} \int_0^T (t-s)^{-\alpha} u(s) ds 
\\
&+ \frac{p}{\Gamma(1-\alpha)} \int_0^T e^{-pt} \left( \int_0^t (t-s)^{-\alpha} u(s) ds \right) dt
=I_1+I_2.
\end{align*}
Since $|u(t)| \le C_0 e^{C_0t}$ for $t\ge0$ with some constant $C_0>0$, we estimate
\begin{align*}
|I_1| 
\le& Ce^{-pT} \int_0^T (T-s)^{-\alpha} e^{C_0s} ds
= Ce^{-pT} \int_0^T s^{-\alpha} e^{C_0(T-s)} ds
\\
=& Ce^{-(p-C_0)T} \int_0^T s^{-\alpha} e^{-C_0s} ds
\le Ce^{-(p-C_0)T} \int_0^\infty s^{-\alpha} e^{-C_0s} ds
= Ce^{-(p-C_0)T} \frac{\Gamma(1-\alpha)}{C_0^{1-\alpha}}.
\end{align*}
Hence if $p>C_0$, then $\lim_{T\to\infty} I_1 = 0$.

As for $I_2$, by the Fubini lemma, we see that
\begin{align*}
I_2 
=& \frac{p}{\Gamma(1-\alpha)} \int_0^T \left( \int_s^T e^{-pt} (t-s)^{-\alpha} dt \right) u(s) ds
\\
=& \frac{p}{\Gamma(1-\alpha)} \int_0^T \left( \int_0^{T-s} e^{-p\eta} \eta^{-\alpha} d\eta \right) e^{-ps} u(s) ds.
\end{align*}
For $p>C_0$, since $|u(s)|\le Ce^{C_0s}$ for some $s\ge0$, we have
$$
\left| \int_0^{T-s} e^{-p\eta} \eta^{-\alpha} d\eta e^{-ps} u(s)  \right|
\le C\left(\int_0^\infty e^{-p\eta} \eta^{-\alpha} d\eta \right) e^{-(p-C_0)s}
$$
for all $s>0$ and $T>0$, and the Lebesgue dominated convergence theorem yields 
\begin{align*}
\lim_{T\to\infty} I_2 
=& \frac{p}{\Gamma(1-\alpha)} \int_0^\infty \left( \int_0^\infty e^{-p\eta} \eta^{-\alpha} d\eta \right) e^{-ps} u(s) ds
\\
=& \frac{p}{\Gamma(1-\alpha)} \frac{\Gamma(1-\alpha)}{p^{1-\alpha}} \int_0^\infty  e^{-ps} u(s) ds 
=p^\alpha Lu(p) 
\end{align*}
for $p>C_0$. Thus the proof of the lemma is complete.
\end{proof}

\subsection{Some results from spectral theory}

We define the the operator $D_m^k$, $k\in\mathbb N$, related to the eigenvalue 
$\lambda_m$ of the operator $-A$ as follows
$$
D_m^k \varphi := \frac1{2\pi i} \int_{\gamma_m} (\eta - \lambda_m)^k 
(\eta - A)^{-1} \varphi d\eta,\quad \varphi\in L^2(\Omega),
$$
where $\gamma_m$ is a sufficiently small circle surrounding the eigenvalue 
$\lambda_m$ of the operator $-A$ (e.g., Kato \cite{Ka}). 
From the result from Suzuki and Yamamoto 
\cite{SYam}, we see that the multiplicity of the eigenvalue $\lambda_m$ is finite and we assume it as $m_\lambda$. 
Then we have
$$
D_m^k = 0 \mbox{ in $L^2(\Omega)$},\quad k\ge m_\lambda.
$$
We call $P_m:=D_m^0$ is the eigenprojection related to the eigenvalue $\lambda_m$ of the operator $-A$, and we see that
\begin{lem}
\label{lem-Pm}
Let $k_0$ be a positive integer. If $\varphi \in L^2(\Omega)$ satisfies $D_m^{k_0} P_m \varphi = 0$, then 
$$
D_m^{k_0-1} P_m \varphi \in {\rm Ker}(\lambda_m - A).
$$
\end{lem}
\begin{proof}
Since $P_m \varphi \in \mathcal D(A)$, we see that $A (\eta - A)^{-1} P_m \varphi = (\eta - A)^{-1} A P_m \varphi$ for any $\varphi\in L^2(\Omega)$, hence that
\begin{align*}
(\lambda_m - A) D_m^{k_0 - 1} P_m \varphi 
= \frac1{2\pi i} \int_{\gamma_m} (\eta - \lambda_m)^{k_0 - 1}(\eta - A)^{-1} (\lambda_m - A) P_m\varphi d\eta
\end{align*}
By writting $\lambda_m - A = \lambda_m - \eta + \eta - A$, we have
\begin{align*}
(\lambda_m - A) D_m^{k_0 - 1} P_m \varphi 
=-\frac1{2\pi i} \int_{\gamma_m} (\eta - \lambda_m)^{k_0}(\eta - A)^{-1} P_m\varphi d\eta
+\frac1{2\pi i} \int_{\gamma_m} (\eta - \lambda_m)^{k_0 - 1} P_m\varphi d\eta.
\end{align*}
In view of the assumption $D_m^{k_0} = 0$ and the residue theory, it follows that the two terms on the right-hand side of the above equation are zero, that is, 
$$
D_m^{k_0 - 1} P_m \varphi \in {\rm Ker}(\lambda_m - A).
$$ 
This completes the proof.
\end{proof}

\section{Proof of Theorem \ref{thm-ucp}}
\label{sec-ucp}
This section is devoted to the proof of the first main result, Theorem \ref{thm-ucp}.
Before giving the proof, we first employ the Laplace transform treatment to show the uniqueness in determining the Neumann derivative of the initial value from the addition data of the solution on the subboundary , which plays crucial role in the proof of Theorem \ref{thm-ucp}. We have
\begin{lem}
\label{lem-Dm}
Assume $u_0\in L^2(\Omega)$ and $u\in L^2(0,T;H^2(\Omega)\cap H_0^1(\Omega))$, $u-u_0\in H_{\alpha}(0,T;L^2(\Omega))$ solves the initial-boundary value problem \eqref{eq-gov}.
If $\partial_{\nu_A} u = 0$ on $\Gamma\times(0,T)$, then for any $m,k\in\mathbb N$, $\partial_{\nu_A} D_m^k u_0 = 0$ on the subboundary $\Gamma$.
\end{lem}
\begin{proof}
From the $t$-analyticity of the solution stated in Lemma \ref{lem-analy}, we can make unique extension for $u(x,t)$, $t\in(0,T)$ to $(0,\infty)$. Therefore, taking Laplace transforms on both sides of \eqref{eq-gov} implies
\begin{equation}
\label{eq-lap}
\left\{
\begin{alignedat}{2}
& A \widehat u(s) + s^{\alpha } \widehat u(s) = s^{\alpha -1} u_0
&\quad& \mbox{in $\Omega$,}
\\
&\widehat u(s)|_{\partial\Omega}=0, &\quad& \Re s\ge s_0
\end{alignedat}
\right.
\end{equation}
together with the formula from Laplace transform in Lemma \ref{lem-lap-caputo}.

Therefore for $s^{\alpha }$ in the resolvent set $\rho(A)$ of the operator $A$, we see that
$$
\widehat u(s) = s^{\alpha -1}  \left(s^{\alpha } + A\right)^{-1} (u_0).
$$
Moreover, the assumption $\partial_{\nu_A}u=0$ on $\Gamma\times(0,T)$ combined  with the $t$-analyticity of the solution, it follows that
$$
\partial_{\nu_A} \widehat u(s) = 0 \quad \mbox{on $\Gamma$}.
$$
Now letting $\eta:=-s^{\alpha }$, we conclude from the above equality that 
$$
\partial_{\nu_A} (\eta-A)^{-1} u_0 = 0 \quad \mbox{on $\Gamma$.}
$$
for any $\eta\in \rho(A)$, from which we further verify
$$
\frac1{2\pi i} \int_{\gamma_m} (\eta - \lambda_m)^k \partial_{\nu_A}(\eta - A)^{-1} u_0 d\eta = 0,
$$
that is, $\partial_{\nu_A} D_m^k u_0 = 0$ in view of the definition of the operator $D_m^k$. This completes the proof of the lemma.
\end{proof}

Now we are ready for the proof of our first main result.
\begin{proof}[Proof of Theorem \ref{thm-ucp}]
From Lemma \ref{lem-Pm} we derive
$$
(\lambda_m - A)(D_m^{m_\lambda-1}P_mu_0) = 0\quad \mbox{in $\Omega$.}
$$
Moreover, since $D_m^k P_m u_0\in H^2(\Omega)\cap H_0^1(\Omega)$ and $\partial_{\nu_A} D_m^k P_m u_0=0$ on $\Gamma$, $k=1,2,\cdots$, we conclude from the unique continuation principle for the elliptic equations that $D_m^{m_\lambda-1}P_m u_0 = 0$ in $\Omega$. Again similar argument yields $D_m^{m_\lambda-2}P_m u_0=0$ in $\Omega$. Continuing this procedure, we obtain $P_m u_0=0$ in $\Omega$ for any $m\in\mathbb N$. Therefore we must have $u_0=0$ from 
the completeness of the generalised eigenfunctions (see the last chapter of \cite{Ag}).

Finally, from the uniqueness for the initial-boundary value problem \eqref{eq-gov}, it follows that $u\equiv0$. This completes the proof of our first main theorem.
\end{proof}

\section{Proof of Theorem \ref{thm-isp}}
\label{sec-isp}
Now let us turn to the proof of the uniqueness of the inverse source problem. The argument is mainly based on the weak unique continuation and the following Duhamel's principle for time-fractional diffusion equations.

\begin{lem}[Duhamel's principle]
\label{lem-Duhamel}
Let $f\in L^2(\Omega)$ and $\mu\in C^1[0,T]$. Then the weak solution $y$ to 
the initial-boundary value problem \eqref{eq-sp} allows the representation
\begin{equation}
\label{eq-Duhamel}
y(\cdot,t)=\int_0^t\theta(t-s)\,v(\,\cdot\,,s) ds,\quad 0<t<T,
\end{equation}
where $v$ solves the homogeneous problem
\begin{equation}
\label{equ-homo}
\begin{cases}
\partial_t^\alpha v + A v=0 & \mbox{in }\Omega\times(0,T),\\
v=f & \mbox{in }\Omega\times\{0\},\\
v=0  & \mbox{on }\partial\Omega\times(0,T)
\end{cases}
\end{equation}
and $\theta\in L^1(0,T)$ is the unique solution to the fractional integral equation
\begin{equation}\label{eq-FIE-te}
J^{1-\alpha}\theta(t)=\mu(t),\quad 0<t<T.
\end{equation}
\end{lem}

The above conclusion is almost identical to Liu, Rundell and Yamamoto \cite[Lemma 4.1]{LRY15} for the single-term case and Liu \cite[Lemma 4.2]{L15} for the multi-term case, except for the existence of non-symmetric part. Since the same argument still works in our setting, we omit the proof here.

\begin{proof}[Proof of Theorem \ref{thm-isp}]
Let $y$ satisfy the initial-boundary value problem \eqref{eq-sp} with $f(x)\,\mu(t)$, where $f\in H_0^1(\Omega)$ and $\mu\in C^1[0,T]$. Then $y$ takes the form of \eqref{eq-Duhamel} according to Lemma \ref{lem-Duhamel}. Performing the Riemann-Liouville fractional integral $J^{1-\alpha}$ to \eqref{eq-Duhamel}, we deduce
\begin{align*}
J^{1-\alpha}u(\,\cdot\,,t) & =\frac1{\Gamma(1-\alpha)}\int_0^t\frac1{(t-\tau)^{\alpha}}\int_0^\tau\theta(\tau-\xi)\,v(\,\cdot\,,\xi)\, d \xi d \tau\\
& =\frac1{\Gamma(1-\alpha )}\int_0^tv(\,\cdot\,,\xi)\int_\xi^t\frac{\theta(\tau-\xi)}{(t-\tau)^{\alpha }}\, d \tau d \xi\\
& =\int_0^tv(\,\cdot\,,\xi)\frac1{\Gamma(1-\alpha )}\int_0^{t-\xi}\frac{\theta(\tau)}{(t-\xi-\tau)^{\alpha }}\, d \tau d \xi\\
& =\int_0^tv(\,\cdot\,,\xi)J^{1-\alpha }\theta(t-\xi)\, d \xi=\int_0^t\mu(t-\tau)\,v(\,\cdot\,,\tau)\, d \tau,
\end{align*}
where we applied Fubini's theorem and used the relation \eqref{eq-FIE-te}. Then the vanishment of $\partial_{\nu_A}u$ on $\Gamma\times(0,T)$ immediately yields
\[
\int_0^t\mu(t-\tau)\partial_{\nu_A}v(\,\cdot\,,\tau)\, d \tau=0\quad\mbox{on }\Gamma,\ 0<t<T.
\]
Differentiating the above equality with respect to $t$, we obtain
\[
\mu(0)\partial_{\nu_A}v(\,\cdot\,,t)+\int_0^t\mu'(t-\tau)\partial_{\nu_A}v(\,\cdot\,,\tau) d \tau=0,\quad\mbox{on }\Gamma,\ 0<t<T.
\]
Owing to the assumption that $|\mu(0)|\ne0$, we estimate
\begin{align*}
\|\partial_{\nu_A}v(\,\cdot\,,t)\|_{L^2(\Gamma)} & \le\frac1{|\mu(0)|}\int_0^t|\mu'(t-\tau)|\|\partial_{\nu_A}v(\,\cdot\,,\tau)\|_{L^2(\Gamma)}\, d \tau
\\
& \le\frac{\|\mu\|_{C^1[0,T]}}{|\mu(0)|} \int_0^t\|\partial_{\nu_A}v(\,\cdot\,,\tau)\|_{L^2(\Gamma)}\, d \tau,\quad 0<t<T.
\end{align*}
Taking advantage of Gronwall's inequality, we conclude $\partial_{\nu_A}v=0$ on $\Gamma\times(0,T)$. Finally, we apply Theorem \ref{thm-ucp} to the homogeneous problem \eqref{equ-homo} to derive $v=0$ in $\Omega\times(0,T)$, implying $f=v(\,\cdot\,,0)=0$. This completes the proof of Theorem \ref{thm-isp}.
\end{proof}

\section{Concluding remarks}
\label{sec-rem}

In this paper, we considered the multi-term time-fractional diffusion equation with advection. By taking Laplace tranform argument, we changed the problem \eqref{eq-gov} to an elliptic equation in the frequency domain. We then proved the weak unique continuation property of the solution to \eqref{eq-gov} by using the spectrol decomposition of the general operator and unique continuation for the elliptic equation. The statement concluded in Theorem \ref{thm-ucp} will be called as the weak unique continuation property because we impose the homogeneous Dirichlet boundary condition on the whole boundary, which is absent in the usual parabolic prototype. (see Cheng, Lin and Nakamura \cite{CLN}, Lin and Nakamura \cite{LN} and Xu, Cheng and Yamamoto \cite{XCY}). As a direct conclusion of the weak unique continuation, we proved that the uniqueness in determining the source term from the boundary measurement. 

Let us mention that the argument used for the proof of the weak unique continuation principle heavily relies on the choice of the coefficients of the fractional derivatives, namely constant coefficients. 
It would be interesting to investigate what happens if this assumption is not valid.
On the other hand, we mention that in the one-dimensional case, the unique continuation (not weak type) for the fractional diffusion equation is valid, one can refer to the recent work from Li and Yamamoto \cite{LY19} in which the theta function method and Phragm\'em-Lindel\"of principle play essential roles in the proof.  Unfortunately, the technique used in \cite{LY19} cannot work in showing the unique continuation for the fractional diffusion equation in the general dimensional case due to the absence of the theta function. This is one of reasons why the unique continuation in the general case was only established in the weak sense. 

To sum up, to overcome the above conjuecture, a new approach may need to be constructed rather than the Laplace transform and spectral decomposition.

\section*{Acknowledgement}
The first author is supported by National Natural Science Foundation of China (No. 11871240) and self-determined research funds of CCNU from the colleges’ basic research and operation of MOE (No. CCNU20TS003).
The second author thanks National Natural Science Foundation of China (No. 11801326). The third author thanks the ENS Rennes and the AMOPA Section d'Ille-et-Vilaine (35) for their financial support.
The fourth author is supported by Grant-in-Aid for Scientific Research
(S) 15H05740 of Japan Society for the Promotion of Science, NSFC (No.
11771270, 91730303) and the \lq\lq RUDN University Program 5-100\rq\rq. This work was also supported by A3 Foresight Program \lq\lq Modeling and Computation of Applied Inverse Problems\rq\rq of Japan Society for the Promotion of Science.


\end{document}